\documentclass[11pt,reqno]{amsart}
\usepackage{amssymb,amsmath,amsthm,newlfont,enumerate}

\newtheorem{prop}{Proposition}[section]

\newtheorem{thm}[prop]{Theorem}
\newtheorem{lem}[prop]{Lemma}
\newtheorem{cor}[prop]{Corollary}

\newtheorem{rem}[prop]{Remark}
\newcommand{\cB}{{\cal B}}

\newcommand{\iI}{{\cal I}}

\newcommand{\BB}{{\mathbb B}}
\newcommand{\DD}{{\mathbb D}}

\newcommand{\TT}{{\mathbb T}}

\DeclareMathOperator{\dist}{dist}

\renewcommand{\ker}{\mathrm{Ker\,}}
\newcommand{\Clos}{\mathrm{clos\,}}

\newcommand{\hH}{\mathrm{H^2 }}
\newcommand{\kK}{\mathrm{K_\theta }}
\newcommand{\lL}{\mathrm{L^2 }}
\newcommand{\kH}{\mathrm{H^\infty   }}
\newcommand{\kL}{\mathrm{L^\infty }}
\renewcommand{\dist}{\mathrm{dist }}
\newcommand{\supp}{\mathrm{supp\, }}
\newcommand{\aA}{{\cal{A}(\DD)\,}}

\title[Compact operators]{Compact operators that commute with a contraction}

\author[kellay]{K. Kellay}
\address{CMI\\LATP\\Universit\'e de Provence\\
39, rue F. Joliot-Curie\\13453 Marseille cedex 13\\France}
\email{kellay@cmi.univ-mrs.fr}

\author[zarrabi]{M. Zarrabi}
\address{ Universit\'e de Bordeaux, UMR 5251 \\
351, cours de la Lib\'eration\\
F-33405 Talence cedex\\ France}
\email{Mohamed.Zarrabi@math.u-bordeaux1.fr}

\thanks{The research of the first author was supported in part by ANR Dynop.}



\begin{document}
\maketitle
\begin{abstract} Let $T$ be a $C_0$--contraction on a separable Hilbert space. We assume that  $I_H-T^*T$ is compact. For a function $f$ holomorphic in 
the unit disk $\DD$ and continuous on $\overline\DD$, we show that $f(T)$ is compact if and only if $f$ vanishes on $\sigma (T)\cap \TT$, 
where $\sigma (T)$ is the spectrum of $T$ and $\TT$  the unit circle. If $f$ is just a bounded holomorphic function on $\DD$ we prove 
that $f(T)$ is compact if and only if $\lim_{n\to \infty} T^nf(T) =0$. 

\end{abstract}

\section{Introduction}

Let  $H$  be a separable Hilbert space, and  $\mathcal{L}(H)$  the space of all bounded operators on  $H$.  For  $T\in\mathcal{L}(H)$, we denote by  $\sigma(T)$ the spectrum of $T$. The Hardy space $\kH $ is the 
set of all bounded and holomorphic functions on $\DD$. The spectrum of an inner function $\theta$  is defined by 
$$
\sigma(\theta)=\Clos \theta^{-1}(0)\cup \supp \mu,
$$ 
where $\mu$ is the singular measure associated to the singular part of $\theta$ and $\supp \mu$ is 
the closed support of $\mu$ (see \cite{Nik}, p. 63). A contraction  $T$ on $H$ is called  a 
$C_0$--contraction (or in class $C_0$) if it is completely nonunitary and there exists a nonzero function $\theta\in \kH$ 
 such that  $\theta(T)=0$, thus there exists a minimal inner function $m_T$ that annihilates $T$, i.e $m_T(T)=0$, and we have  $\sigma(T)=\sigma(m_T)$ 
 (see  \cite{NF} p. 117 and  \cite{Nik}, p. 71-72).  A contraction $T$ is said essentially unitary if  $I_H-T^*T$ is compact, where $I_H$ 
 is the identity map on $H$.  

Let $T$ be a $C_0$--contraction on  $H$, and let $\kH (T)=\{f(T)\text{ : }f\in \kH \}.$  $\kH (T)$ is clearly a subspace of the commutant $\{T\}'=\{A\in \mathcal{L}(H)\text{ : } AT=TA\}$.  In this note we study the question of when  $\kH (T)$ contains a  nonzero compact operator.  B. Sz--Nagy \cite{Na}, proved that $\{T\}'$ contains always a nonzero compact operator, but there exists a $C_0$--contraction $T$ such that zero is the unique compact operator contained in   $\kH (T)$ (see section 4). Nordgreen \cite{Nor} proved that if $T$ is an essentially unitary $C_0$--contraction  then $\kH (T)$ contains a nonzero compact operator. There are also results about the existence of smooth operators (finite rank, Schatten--von Neuman operators) in $\kH (T)$ (see \cite{V}).
It is also shown in the Atzmon's paper \cite{A}, that if $T$ is a cyclic completely nonunitary contraction such that $\sigma (T)=\{1\}$ and
\begin{equation}
\log \Vert T^{-n}\Vert=O(\sqrt{n}),\ n\to\infty,
\end{equation}
then $T-I_H$ is compact. Our hope in this paper is to establish results of this kind for more general contractions. Section 2 is devoted to the study of the compactness of $f(T)$ when $f$ is in the disk algebra. We recall that the disk algebra $\aA$ is the space of all functions that are holomorphic in $\DD$ and continuous on $\TT$. We show (Theorem \ref{thm1}), that, if $f\in \aA$ and if $T$ is a $C_0$--contraction which is essentially unitary, then $f(T)$ is compact if and only if $f$ vanishes on $\sigma(T)\cap\TT$. The main tool used in the proof of this result is the Beurling-Rudin theorem about the characterization of the closed ideals of $\aA$. For a large class of $C_0$--contractions we show (Proposition \ref{proposition1}) that the condition `` $T$ is essentially unitary'' is necessary in the above result. As corollary, we obtain that if $T$ is a contraction that is annihilated by a nonzero function in $\aA$ and if $T$ is cyclic (or, more generally, of finite multiplicity) then $f(T)$ is compact whenever $f\in \aA$ and $f$ vanishes on $\sigma(T)\cap\TT$. We notice  that an invertible  contraction with spectrum reduced to a single point and satisfying condition (1) is necessarily annihilated by a nonzero function in $\aA$ (see \cite{At}).

In section 3, we are interested in the compactness of $f(T)$ when $f\in \kH$. Let $\hH$ be the usual  Hardy space on $\DD$, $\theta$ being an inner function and  
$\kK= \hH\ominus \theta\hH$. We consider the model operator $T_\theta: \kK\to  \kK$ given by  $T_\theta g=P_\theta(z g)$, where $P_\theta$ stands the 
orthogonal projection on $\kK$.  Notice that  $m_{T_\theta}=\theta$ and the commutant 
$\{T_\theta\}'$ of  $T_\theta$ coincide with  $\kH (T_\theta)$.
Hartman \cite{Har} and Sarason \cite{S} gave a complete characterization of compact operators commuting with $T_\theta$ (see \cite{Nik} p. 182). They showed that 
$$
T^{n}_{\theta} f(T_\theta)\to 0 \Longleftrightarrow f(T_\theta) {\text { is compact } } \Longleftrightarrow f\overline{\theta}\in  \kH + C(\TT).
$$
For $C_{00}$ contractions, that is contractions $T$ such that $T^nx\longrightarrow 0$ and ${T^*}^nx\longrightarrow 0$ for every $x\in H$,  Muhly gave in \cite{Mu} (see also \cite{P}) a characterization of functions $f\in \kH $ such that $f(T)$ is compact, in term of the characteristic function of $T$.
With the help of the corona theorem,  we show (Theorem \ref{muhly}) that if $T$ is an essentially unitary $C_0$-contraction, then $f(T)$ ($f\in\kH$) is compact if and only if $T^{n} f(T)\to 0$. We obtain in particular that if $\lim_{r\to 1-} f(rz)=0$ for every $z\in \sigma (T)\cap\TT$, then $f(T)$ is compact.

\section{ Compactness of $f(T)$ with $f$ in the disk algebra}

Let $T$ be a $C_0$--contraction on  $H$. We will introduce some definitions and results we will need later.
We call $\lambda\in\sigma(T)$ a normal eigenvalue if it is an isolated point of $\sigma (T)$ and if the corresponding Riesz projection has a finite rank. We set $\sigma_{np}(T)$ the set of all normal eigenvalues of $T$. The weakly continue spectrum of $T$ is defined by $\sigma_{wc}(T)=\sigma (T)\setminus \sigma_{np}(T)$ (see \cite{Nik2}, p. 113).

Let us suppose furthermore that $T$ is essentially unitary. Since  $\DD\setminus \sigma (T) \ne \emptyset$, there exists a unitary operator $U$ and a compact operator $K$ such that $T=U+K$ (see \cite{Nik2} p. 115) and we have $\sigma_{wc}(T)=\sigma_{wc}(U)\subset\TT$. Since $T$ is in class $C_0$ the set of eigenvalues of $T$ is $\sigma (T)\cap \DD$ (see \cite{Nik} p. 72). In particular we have $\sigma_{np}(T)\subset \sigma (T)\cap \DD$. It follows that $\sigma_{wc}(T)=\sigma(T)\cap\TT$ and $\sigma_{np}(T)= \sigma (T)\cap \DD$. We deduce also from the above observations that if $T$ is in class $C_0$ then $T$ is essentially unitary if and only if $T^*$ is too.

Let $\iI$ be a closed ideal of $\aA$. We
denote by $S_{\iI}$ the inner factor of $\iI$, that is the greatest inner common divisor  of all nonzero functions in $\iI$ (see \cite{H} p. 85), and we set $Z(\iI)=\bigcap_{f\in \iI}\{\zeta\in \TT\text{ : } f(\zeta)=0\}$. For $E\subset\TT$ we set 
$ \mathcal{J}(E)=\{f\in \aA\text{ : } f_{|E}=0\}$. We shall need the Beurling-Rudin theorem (see \cite{H} p. 85) about the structure  of closed ideals of  $\aA$, which states  that every closed ideal $\iI\subset\aA$ has the form 
$$
\iI=  S_{\iI}\kH \cap \mathcal{J}\big(Z(\iI)\big).
$$

\begin{thm}{\label{thm1}} Let $T$ be an essentially unitary  $C_0$--contraction and let  $f\in \aA$. The following assertions are equivalents. 
\begin{enumerate}
\item $f(T)$ is compact.
\item $f=0 {\text { on }} \sigma (T)\cap \TT$.
\end{enumerate}
\end{thm}

For the proof of this theorem we need the following lemma.

\begin{lem}\label{com}
 Let $T_1,T_2$ be two contractions on $H$ such that $T_1-T_2$ is compact and $f\in A(\DD)$. Then $f(T_1)$ is compact if and only if  $f(T_2)$ is too.
 \end{lem}
 \begin{proof}
 There exists a sequence  $(P_n)_n$ of polynomials such that $\|f-P_n\|_{_\infty}\to 0$, where $\| . \|_\infty$ is the supremum norm on $\TT$. We set $R_n=f-P_n$. Note that for every $n$,   $P_n(T_2)-P_n(T_1)$ is compact. On the other hand, for $i=1,2$, by the  von Neumann  inequality, we have  $\|R_n(T_i)\|\leq \|R_n\|_{_\infty}$. So $\|R_n(T_i)\|\longrightarrow 0$. It follows that 
$$
\begin{aligned}
f(T_2)-f(T_1) & = \lim_{n\to +\infty} \big(P_n(T_2)-P_n(T_1)+R_n(T_2)-R_n(T_1)\big) \\
& =\lim_{n\to +\infty} \big(P_n(T_2)-P_n(T_n)\big).
\end{aligned}
$$
Thus $f(T_2)-f(T_1)$ is compact, which finishes the proof.
 \end{proof}

{\bf Proof of Theorem \ref{thm1}.}
$(1) \Rightarrow (2):$ Let  $\cB_T$ denote the maximal commutative Banach algebra that contains $I_H$ and $T$. We have   $\sigma (T)=\sigma_{\cB_T} (T)$, where 
$\sigma_{\cB_T} (T)$ is the spectrum of  $T$ in $\cB_T$. Let  $\lambda\in\sigma (T)$, there exists  a character $\chi_\lambda$
on  $\cB_T$ such that $\chi_\lambda (T)=\lambda$. Since the set of polynomials is dense in $\aA$, 
 $$\chi_\lambda
(f(T))=f(\chi_\lambda (T))=f(\lambda), \qquad f\in\aA.$$
Let now $f\in \aA$ such that $f(T)$ be compact and let $\lambda\in\sigma (T)\cap\TT$. We have
\begin{equation}\label{formule1}
|f(\lambda)|=|\lambda^nf(\lambda)|=|\chi_\lambda (T^nf(T))|\leq 
\Vert T^nf(T)\Vert.
\end{equation}
Since $T$ is in class $C_0$, $T^nx\to 0$ whenever  $x\in H$, (see \cite{NF} Proposition III.4.1).  Thus for every compact set $C\subset H$, 
$$
\lim_{n\to\infty}\sup_{x\in C}\|T^nx\|=0.
$$
For $C=\overline{f(T)(\BB)}$, where $\BB=\{x\in H\text{ : } \|x\|\leq 1\}$, we get $T^nf(T)\to 0$. Then it follows  from (\ref{formule1}) that $f(\lambda)=0$.

$(2)\Rightarrow (1):$ Without loss of generality, we may assume that  $\sigma(T)\cap\TT$ is of Lebesgue measure zero.  We set
$$\iI=\{f\in \aA\text{ : } f(T)\text{ compact } \};$$
$\iI$ is a closed ideal of  $\aA$. We have to prove that $\iI= \mathcal{J}(\sigma(T)\cap \TT)$. By the Beurling--Rudin theorem, it suffice to show that $Z(\iI)=\sigma(T)\cap \TT$ and $S_{\iI}=1$. In the proof of the implication ($(1)\Rightarrow (2)$) we have seen that
$\iI\subset   \mathcal{J}(\sigma(T)\cap \TT)$, which implies  $\sigma(T)\cap \TT\subset Z(\iI)$. It remains to show that  $S_{\iI}=1$ and $Z(\iI)\subset \sigma(T)\cap \TT$. 

Since $T$ is in class $C_0$ and $I_H-T^*T$ is compact, by  the observation in the beginning of this section  $T=U+K$, where $U$ is unitary and  $K$ is compact. Moreover we have $\sigma_{wc}(U)=\sigma_{wc}(T)=\sigma(T)\cap \TT$ (\cite{Nik2} p. 115), and since 
 $\sigma_{np}(U)$ is countable, we see that $\sigma (U)$ is a subset of $\TT$ of Lebesgue measure zero. By the Fatou theorem (\cite{H} p. 80), there exists a 
 nonzero outer function $f\in \aA$ which vanishes exactly on $\sigma(U)$. We have $f(U)=0$ since $U$ is unitary. By Lemma \ref{com}, $f(T)$ is compact. This shows that  $S_{\iI}=1$ and $Z(\iI)\subset \sigma(U)$. We shall now show that $Z(\iI)\subset \sigma_{wc}(U) $.

Let $\lambda\in \sigma_{np}(U)$; $\lambda$ is an isolated point in  $\sigma(U)$ and $\ker (U-\lambda I_H)$ is of finite dimension. There exists $g\in \aA$ with $g(\lambda)\neq 0$ and $f_{|\sigma(U)\backslash\{\lambda\}}=0$. Since $(z-\lambda)f(z)=0$, $z\in \sigma(U)$ and $U$ unitary, $(U-\lambda)f(U)=0$. So $ f(U)(H)\subset \ker (U-\lambda I_H)$. So $f(U)$ is of finite rank, thus $f(U)$ is compact and by Lemma \ref{com}, $f(T)$ is compact. Hence $\lambda\not\in Z(\iI)$. We deduce that $Z(\iI)\subset \sigma_{wc}(U)=\sigma_{wc}(T)=\sigma(T)\cap\TT$, which finishes the proof.

\vspace{1em}

Let $T\in\mathcal{L}(H)$. The spectral multiplicity of $T$ is the cardinal number  given by the formula
$$
\mu_T=\inf {\text { card } L},
$$
where ${\text { card } L}$ is the cardinal of $L$ and where the infimum is taken over all nonempty sets $L\subset H$ such that $\text{span}\{T^nL; \ n\geq 0\}$ is dense in $H$. Notice that $\mu_T=1$ means that $T$ is cyclic.

\begin{cor}\label{corollaire} Let  $T$ be a contraction on $H$ with  $\mu_T<+\infty$.  Assume that there exists a nonzero function $\varphi \in \aA$ such that $\varphi (T)=0$. Then $f(T)$ is compact for every function $f\in \aA$ that vanishes on $\sigma(T)\cap \TT$.
\end{cor}

\begin{proof} There exists two orthogonal Hilbert subspaces $H_u$ and $H_0$ that are invariant by $T$, such that $H=H_u\oplus H_0$, $T_u=T_{|H_u}$ is unitary and $T_0=T_{|H_0}$ is completely nonunitary (see \cite{NF}, Theorem 3.2, p. 9 or \cite{Nik}, p. 7).  $T_0$ is clearly in class $C_0$ and we have $\mu_{T_0}<+\infty$. By Proposition 4.3 of \cite{BV}, $I_{H_0}-T_0^*T_0$ is compact. 

Let $f\in \aA$, with $f_{|\sigma(T)\cap \TT}=0$. Since $\sigma (T_0)\subset \sigma (T)$, it follows from Theorem \ref{thm1} that $f(T_0)$ is compact. Now, since $T_u$ is unitary and $\sigma (T_u)\subset \sigma (T)\cap\TT$, we get $f(T_u)=0$. Thus $f(T)$ is compact.
 \end{proof}

\begin{rem} {\rm Let $T$ be a cyclic contraction satisfying condition (1) and with finite spectrum, $\sigma (T)=\{\lambda_1,\ldots,\lambda_n\}$. By Theorem 2 of \cite{At}, there exist a function $f=\sum_{n\geq 0} a_n z^n$, $f\ne 0$, such that $\sum_{n\geq 0} |a_n|< +\infty$ and $f(T)=0$. Then, it follows  from Corollary \ref{corollaire} that $(T-\lambda_1)\ldots (T-\lambda_n)$ is compact. Thus we find the result  Corollary 4.3 of \cite{A}, mentioned in the introduction.}
\end{rem}

Now we finish this section by showing  that the hypothesis ''essentially unitary`` in Theorem \ref{thm1} is necessary. Let us first make some observations. An operator $T\in\mathcal{L}(H)$ is called essentially normal if $TT^*-T^*T$ is compact. Notice that if $T$ is a $C_0$--contraction which is essentially unitary then $T^*$ is  essentially unitary too. Hence  $T$ is essentially normal since $I_H-T^*T$ and $I_H-TT^*$ are both compacts. Notice also that Theorem \ref{thm1} is of interest in the case of contractions $T$ such that  $\sigma(T)\cap \TT$ is of Lebesgue measure zero.

\begin{prop}\label{proposition1} Let $T\in\mathcal{L}(H)$ be a $C_0$--contraction which is essentially normal and such that $\sigma(T)\cap \TT$ is of Lebesgue measure zero. Assume that for every $f\in \aA$ vanishing on $\sigma(T)\cap \TT$, $f(T)$ is compact. Then $T$ is essentially unitary.
 \end{prop}

\begin{proof}  Denote by $\mathcal{K}(H)$ the set of all compact operators on $H$ and by $\pi\ : \mathcal{L}(H) \longrightarrow \mathcal{L}(H)/\mathcal{K}(H)$  the canonical surjection. The essential spectrum $\sigma_{ess}(T)$ of $T$ is defined as the spectrum of $\pi (T)$ in the Banach algebra $\mathcal{L}(H)/\mathcal{K}(H)$.

Let $\lambda\in \sigma(T)\cap \DD$. By Fatou theorem \cite{H}, there exists a non zero  outer function $f\in \aA$ such that $f_{|\sigma(T)\cap \TT}=0$. By hypothesis $f(T)$ is compact. The function $z-\lambda$ and $f$ have no common zero in $\overline\DD$. So there exists two functions $g_1$ and $g_2$ in $\aA$ such that $(z-\lambda)g_1+fg_2=1$. Thus $(T-\lambda)g_1(T)+f(T)g_2(T)=I_H$, which shows that $\pi (T)-\lambda$ is invertible in $\mathcal{L}(H)/\mathcal{K}(H)$. Hence
$\sigma_{ess}(T)\subset \sigma(T)\cap \TT$.

By Rudin-Carleson-Bishop theorem  (see \cite{H} p. 81), there exists a function $h\in \aA$ such that $\overline z=h(z), \ z\in \sigma(T)\cap \TT$. Since $\pi (T)$ is a normal element in the $C^*$ algebra $\mathcal{L}(H)/\mathcal{K}(H)$, we get $\pi (T)^*=h(\pi (T))$. On the other hand  we have $1-h(z)z=0$ on $\sigma(T)\cap \TT$, which implies that $\pi (I_H)-\pi (T)^*\pi (T)=\pi (I_H)-h(\pi (T))\pi (T)=0$. Therefore $I_H-T^*T$ is compact.
 \end{proof}

\section{The case of $f(T)$ for $f\in \kH$}

In this section we are interested in the  compactness of $f(T)$ when $f\in \kH$. As a consequence of Theorem \ref{thm1} we prove the following result which was 
first established by Moore--Nordgren in \cite {MN}, Theorem 1. The proof given in \cite {MN}  uses a result of  Muhly \cite{Mu} (see Remark 1 below),  we give here a simple proof.

\begin{lem}\label{lemme2} Let  $T$ be an essentially unitary  $C_0$--contraction on $H$, $\theta$ be an inner function that divide  $ m_T$ 
(i.e $ m_T/\theta\in \kH$) and such that $\sigma(\theta)\cap \TT$ is of Lebesgue measure zero.   Let   $\psi\in \aA$  be such that 
$\psi_{|\sigma(\theta)\cap \TT} =0$. If $\phi=\psi m_T/ \theta$, then $\phi(T)$ is compact.

In particular the commutant $\{T\}'$ contains a nonzero compact operator.
\end{lem}

\begin{proof}
Let $\Theta= m_T/ \theta$ and $T_1=T|_{\overline{\Theta(T) H}}$ be the restriction of  $T$ to $\overline{\Theta(T) H}$; $T_1$ is a  $C_0$--contraction with  $m_{T_1}=\theta$. Moreover    $I_{H_1}-T^{\ast}_1T_1=P_{H_1}(I_H-T^\ast T)_{|H_1}$ is compact, where  $P_{H_1}$ is the orthogonal projection from $H$ onto $H_1$. By Theorem  \ref{thm1},  $\psi(T_1)$ is compact and thus
$\phi(T)=\psi(T)\Theta(T)=\psi(T_1)\Theta(T)$ is also compact.
\end{proof}

\begin{lem}\label{lemme3} Let  $T$ be an essentially unitary  $C_0$--contraction on $H$, $\theta$ be an inner function that divide  $ m_T$  and such 
that $\sigma(\theta)\cap \TT$ is of Lebesgue measure zero.   Let   $f\in \kH$  be such that $\lim_{n\to +\infty} T^nf(T) =0$. 
If $\phi=fm_T/ \theta$, then $\phi(T)$ is compact.
\end{lem}

\begin{proof}
By the Rudin-Carleson-Bishop theorem, for every nonnegative integers $n$, there exists $h_n\in \aA$ such that $\overline{z}^n=h_n (z), z\in \sigma(\theta)\cap \TT$ and $\|h_n\|_\infty=1$, where $\|.\|_\infty$ is the supremum norm on $\TT$ (see \cite{H} p. 81). We have, for every $n$, $1-z^nh_n(z)=0, \ z\in \sigma(\theta)\cap \TT$, then by Lemma \ref{lemme2}, 
$(I_H-T^nh_n(T))\big(m_T/ \theta\big)(T)$ is compact. It follows that $\phi(T)-T^nf(T)h_n(T)\big(m_T/ \theta\big)(T)$ is also compact. Since 
$$
\|T^nf(T)h_n(T)\big(m_T/ \theta\big)(T)\|\leq \|T^nf(T)\|\longrightarrow 0,
$$
we deduce that $\phi(T)$ is compact.
\end{proof}

We need the  following lemma about inner functions, which is in fact contained in the proof of the main result of \cite{Nor}. For the completeness we include here its  proof.

\begin{lem}\label{lemme4} Let $\Theta$ an inner function. There exists a sequence $(\theta_n)_n$ of inner functions such that for each $n$, $\theta_n$ divides $\Theta$, $\sigma(\theta_n)\cap \TT$ is of Lebesgue measure zero and for every $z\in \DD$, $\lim_{n\to +\infty} \theta_n(z)=\Theta (z)$.
\end{lem}

\begin{proof}
Let $B_n$ be the Blaschke product constructed with the zeros of $\Theta$ contained in the disk $\{|z|\leq  1-1/n\}$, each zero of $\Theta$ repeated according to its  multiplicity.

Let $\nu$ be the singular measure defining the singular part of $\Theta$. There exists $F\subset \TT$ of Lebesgue measure Zero  such that $\nu (F)=\nu (\TT)$. There exists a sequence $(K_n)_n$ of compact subsets of $F$ such that $\lim_{n\to\infty}\nu(K_n)=\nu(F)$. For every $n$, let $\nu_n$ be the measure on $\TT$ defined by  $\nu_n(E)=\nu(E\cap K_n)$. Denote by $S_n$  the singular inner function associated to the measure $\nu_n$. It suffice now to take $\theta_n=B_nS_n$.
\end{proof}

We are now able to prove the main result of this section.

\begin{thm}\label{muhly} Let  $T$ be an essentially unitary  $C_0$--contraction on $H$. Let
$f\in \kH$. Then the following assertions are equivalents.
\begin{enumerate}
\item $T^nf(T)\to 0$.
\item $f(T)$ is  compact.
%\item $f\overline{ m_T}\in  \kH + C(\TT)$.
\end{enumerate}
\end{thm}
\begin{proof}
 $(1)\Rightarrow(2):$ Let $\Theta=m_T$ and let $(\theta_n)_n$ be the sequence of inner functions given by Lemma \ref{lemme4}. 
For every $n$, we set $\varphi_n=m_T/\theta_n$. Since $(\varphi_n)_n$ is a bounded sequence in $\kH$ and $\varphi_n (z)  \longrightarrow 1\ (z\in\DD)$, $(\varphi_n)_n$ converges to 1 uniformly on the compacts of $\DD$. Then for every $k$, there exists a nonnegative integer $n_k$ such that $|\varphi_{n_k}(z)|\geq e^{-1}$ for $|z|\leq {\frac{k}{k+1}}$. Clearly the sequence $(n_k)_k$ may be chosen to be strictly increasing. Moreover for $|z|\geq  {\frac{k}{k+1}}$, we have $|z^k|\geq e^{-1}$. So
$$
e^{-1}\leq |z^k|+|\varphi_{n_k}(z)|\leq 2,\ z\in\DD.
$$
By he corona theorem (\cite{Nik}, p. 66), there exists two functions $h_1$ and $h_2$ in $\kH$ such that
$$
z^kh_1+\varphi_{n_k}h_2=1\ \text{ and } \ |h_1|, |h_2|\leq C,
$$
where $C$ is an absolute constant. Thus we get
$$
T^kf(T)h_1(T)+ f(T)\varphi_{n_k}(T)h_2(T)=f(T),
$$
and 
$$
\|T^kf(T)h_1(T)\|\leq C\|T^kf(T)\|\longrightarrow 0.
$$
Consequently, $f(T)=\lim_{k\to\infty}f(T)\varphi_{n_k}(T)h_2(T)$ in the $\mathcal{L}(H)$ norm. Finally $f(T)$ is compact since by Lemma \ref{lemme3}, for every $k$, $f(T)\varphi_{n_k}(T)h_2(T)$ is compact.

$(2)\Rightarrow(1):$ see the proof of Theorem \ref{thm1}.
\end{proof}

 Let $T$ be a contraction on $H$. It is shown by Esterle, Strouse and Zouakia in \cite{ESZ}, that if  $f\in \aA$, then $\lim_{n\to\infty}T^nf(T)=0$ if and only if $f$ vanishes on $\sigma(T)\cap \TT$. So Theorem \ref{muhly} implies Theorem \ref{thm1}. Now, if $T$ is completely non unitary, Bercovici showed in \cite{B} that if $f\in \kH$ and $\lim_{r\to 1-} f(rz)=0$, for every $z\in \sigma (T)\cap \TT$, then $\lim_{n\to \infty} T^n f(T)=0$. So it follows immediately from this fact and Theorem \ref{muhly} the following result.

\begin{cor} Let  $T$ be an essentially unitary  $C_0$--contraction on $H$. Let
$f\in \kH$. If for every $z\in \sigma (T)\cap \TT$, $\lim_{r\to 1-} f(rz)=0$, then $f(T)$ is compact.
\end{cor}

\section{Remarks}

1. As in Corollary \ref{corollaire}, Theorem \ref{muhly} holds for  a $C_0$-contraction such that $\mu_T<+\infty$.

\vspace{0.5em}

2.  Let  $\mathcal{H}$ be a separable Hilbert space and  $\lL(\mathcal{H})$ the space of weakly measurable functions, $\mathcal{H}$--valued, norm  square 
integrable  functions on $\TT$. We denote by  $\hH(\mathcal{H})$ the space of functions in $\lL(\mathcal{H})$  whose negatively indexed Fourier 
 coefficients vanish.  The space  of  bounded weakly measurable function on $\TT$ taking values in $ \mathcal{L}(\mathcal{H})$ is $\kL (\mathcal{L}(\mathcal{H}))$, 
 and the subspace of $\kL (\mathcal{L}(\mathcal{H}))$ consisting of those functions whose negatively indexed Fourier coefficients vanish is  
 $ \kH (\mathcal{L}(\mathcal{H}))$. A function $\Theta \in \kH (\mathcal{L}(\mathcal{H}))$  is called inner if it is unitary  valued almost everywhere.  
 The shift  $S$ on  $\hH(\mathcal{H})$ is given by  $S F(z)= zF(z)$.  If $\Theta\in \kH (\mathcal{L}(\mathcal{H}))$ is an inner function , 
 $\Theta \hH(\mathcal{H})$ is an invariant subspace for $S$.  Let  $\mathcal{K}(\Theta)$ be the orthogonal complement of  
 $\Theta \hH(\mathcal{H})$ and  $P_{\mathcal{K}(\Theta)}$ be the orthogonal projection on  ${\mathcal{K}(\Theta)}$. The compression  
 $S(\Theta)$ of the shift $S$ to a subspace  $\mathcal{K}(\Theta)$ is given by 
$$S(\Theta)F=P_{\mathcal{K}(\Theta)}(S F),$$

If  $T$ is a $C_0$--contraction on a Hilbert  $H$,  then there exists a Hilbert space $\mathcal{H}$ and  inner functions  
$\Theta\in \kH(\mathcal{L}(\mathcal{H}))$ such that $T$ is unitarily equivalent to  $S(\Theta)$. 

Muhly showed in \cite{Mu} that for  $f\in \kH$, the operator 
  $f(T)$ is compact if and only if 
$$f\Theta^\ast\in \kH(\mathcal{L}(\mathcal{H})) + C(S_\infty(\mathcal{H})),$$ 
where $C(S_\infty(\mathcal{H}))$ is the space of continuous functions on $\TT$ that takes values in $S_\infty(\mathcal{H})$, the space of compact operators 
on $\mathcal{H}$.

Suppose now that the contraction $T$ satisfies the hypothesis of Theorem \ref{muhly}. Since 
$\kH+C(\TT)= \text{Clos }_{_{\kL(\TT)}}\{\overline{z}^n\kH\text{ : } n\geq 0\}$ (\cite{Nik} p. 183), we get for every $f\in\kH$, 
\begin{multline*}
\|T^nf(T)\|=\inf_{h\in\kH} \|T^nf(T)+m_T(T)h(T)\|
 \leq \inf_{h\in\kH} \|z^nf(z)+m_T(z)h(z)\|_\infty\\
=\inf_{h\in\kH}\|f(z)\overline{m_T}(z)+\overline{z}^{n}h(z)\|_\infty=
\dist(f\overline{m_T},\overline{z}^{n}\kH).
\end{multline*}
So if  $f\overline{m_T}\in \kH+C(\TT)$, then $T^nf(T)\to 0$ and by Theorem \ref{muhly}, $f(T)$ is compact. We do not know if the converse is true that is $f(T)$ compact implies that $f\overline{m_T}\in \kH+C(\TT)$.

\vspace{0.5em}

3. The following remark gives an example of a $C_0$--contraction $T$ such that zero is the unique compact operator contained in $\kH (T)$. Let $S$ be any  $C_0$-- contraction on $H$. We consider the Hilbert space 
$$\ell^2(H)=\big\{x=(x_n)_{n\geq 1}\subset H\text{ : }
\| x\|_{2}=\big(\sum_{n\geq 1} \Vert x_n\Vert^2\big)^{1\over2}<+\infty\big\}.$$
and the operator $T$ defined on $\ell^2(H)$ by the formula : $T x= (Sx_n)_n$, $x=(x_n)_n$. Clearly $T$ is a $C_0$-- contraction with $m_T=m_S$. We claim that for every $f\in \kH $ such that $f(T)\ne 0$, $f(T)$ is not  compact. Indeed,  let $v\in H$ with norm 1 and such that  $f(S)v\ne 0$ and for 
$n\geq 1$, let $x(n)\in \ell^2({H})$ defined by  $x(n)_k=v$ if  $k=n$ and $x(n)_k=0$ if $k\ne n$. Since for every $k$ the projection $x\to x_k$ is continue, zero is the unique limit of any convergent subsequence of $\big(f(T)x(n)\big)_n$. On the other hand, for every $n$, $\Vert f(T)x(n)\Vert_2=\Vert f(S)v\Vert$. So all subsequences of $\big(f(T)x(n)\big)_n$ diverge.

\end{document}